\documentclass{amsart}
\usepackage{amsmath,amsthm}
\usepackage{amsfonts,amssymb}
\usepackage{accents}
\usepackage{enumerate}
\usepackage{accents,color}
\usepackage{graphicx}
\usepackage{comment}

\newcommand{\lvt}{\left|\kern-1.35pt\left|\kern-1.3pt\left|}
\newcommand{\rvt}{\right|\kern-1.3pt\right|\kern-1.35pt\right|}

\newtheorem{thm}{Theorem}[section]

\newtheorem{lem}[thm]{Lemma}

\newtheorem{exam}[thm]{Example}

\theoremstyle{remark}
\newtheorem{rem}{Remark}[section]

 \def\la{{\langle}}
 \def\ra{{\rangle}}
 \def\ve{{\varepsilon}}

 \def\d{\mathrm{d}}
 
 \def\i{\mathrm{i}}

 \def\sph{{\mathbb{S}^{d-1}}}

 \def\sd{{\mathsf d}}
 \def\sw{{\mathsf w}}
 \def\sB{{\mathsf B}}

 \def\sW{{\mathsf W}}

 \def\a{{\alpha}}
 \def\b{{\beta}}
 
 \def\k{{\kappa}}
 \def\t{{\theta}}
 \def\l{{\lambda}}

 \def\la{{\langle}}
 \def\ra{{\rangle}}
 \def\ve{{\varepsilon}}

 \def\xb{{\mathbf x}}

 \def\CD{{\mathcal D}}
 
 \def\CH{{\mathcal H}}

 \def\CP{{\mathcal P}}

 \def\CV{{\mathcal V}}

 \def\BB{{\mathbb B}}
 \def\CC{{\mathbb C}}
 
 \def\NN{{\mathbb N}}

 \def\RR{{\mathbb R}}
 \def\SS{{\mathbb S}}

      \def\proj{\operatorname{proj}}

\def\lla{\langle{\kern-2.5pt}\langle}
\def\rra{\rangle{\kern-2.5pt}\rangle}

\def\bk{{\boldsymbol{\kappa}}}

\newcommand{\wh}{\widehat}

\def\f{\frac}

\graphicspath{{./}}
\begin{document}

\title{On Bernstein inequalities on the unit ball}
\author{Tomasz Beberok}
\address{Department of Applied Mathematics, University of Agriculture in Krakow, Poland}
\email{tomasz.beberok@urk.edu.pl}
\author{Yuan~Xu}
\address{Department of Mathematics, University of Oregon, Eugene,
OR 97403--1222, USA}
\email{yuan@uoregon.edu}
\thanks{The first author was supported by the Polish National Science Centre (NCN) Miniatura grant no.
2025/09/X/ST1/00082. The second author was partially supported by Simons Foundation Grant \#849676}
\date{\today}
\subjclass[2010]{33C45, 42C05, 42C10}
\keywords{Bernstein inequality, weighted, $L^p$, unit ball}
\keywords{}

\begin{abstract}
Two types of Bernstein inequalities are established on the unit ball in $\RR^d$, which are stronger than
those known in the literature. The first type consists of inequalities in $L^p$ norm for a fully symmetric doubling
weight on the unit ball. The second type consists of sharp inequalities in $L^2$ norm for the Jacobi weight, which
are established via a new self-adjoint form of the spectral operator that has orthogonal polynomials as eigenfunctions.
\end{abstract}

\maketitle

\section{Introduction}
\setcounter{equation}{0}

A set of new Bernstein inequalities was discovered on the triangle and simplex in \cite{X23, GX}, which are stronger
than the classical result and are somewhat unexpected. Following the lead, we revisit Bernstein inequalities in weighted
$L^p$ norm on the unit ball $\BB^d =\{x: \|x\| \le 1\}$ of $\RR^d$, which have been studied and utilized by many authors;
see \cite{B1, B2, Dai, DaiX1, DaiX, K2009, K2022, K2023, LWW, PX, Wang, X05} and the reference therein.
The classical weight function on the unit ball is of the form
$$
  W_\mu(x) = (1-\|x\|^2)^{\mu}, \qquad x \in \BB^d,  \quad  \mu > -1.
$$
For this radial weight function on the domain that has radial symmetry, several Bernstein inequalities in the literature are
modeled after the classical Bernstein inequality on the interval $[0,1]$, as can be seen, for example, by the inequality
for the $i$-th partial derivative $\partial_i$,
\begin{equation} \label{eq:B1}
  \left \|\varphi^r \partial_i^r f \right \|_{L^p(W_\mu, \BB^d)} \le c \, n^r \|f \|_{L^p(W_\mu, \BB^d)}, \quad 1 \le p \le \infty,
\end{equation}
where $\varphi(x) = \sqrt{1-\|x\|^2}$, $r$ is a positive integer, and $f$ is any polynomial of degree at most $n$. This
appears to be natural and what it should be, since the inequality reduces to, after all, the classical weighted Bernstein
inequality on the interval $[0,1]$ when $d=1$. One of our main results in this work shows, however, the following
inequality holds,
$$
  \left \|\Phi_i ^r \partial_i^r f \right \|_{L^p(W_\mu, \BB^d)} \le c \, n^r \|f \|_{L^p(W_\mu, \BB^d)}, \quad 1 \le p \le \infty,
$$
where $\Phi_i$ is defined by
\begin{equation} \label{eq:B2}
  \Phi_{i}(x) =\frac{ \sqrt{1-\|x\|^2}}{\sqrt{x^2_{i}+1-\|x\|^2}} = \frac{\varphi(x)} {\sqrt{x^2_{i}+1-\|x\|^2}}.
\end{equation}
Since $\varphi(x) \le \Phi_i(x)$ for $x\in \BB^d$, our new inequality is stronger than \eqref{eq:B1}.
{Moreover, for certain monomials, the inequality \eqref{eq:B2} is stronger by a factor of a magnitude
$n^{r/2}$ for $n$ large.} This is surprising,
given how classical approximation theory on the unit ball is, and it is only one of several other inequalities of the
same nature. Moreover, the inequality \eqref{eq:B1} holds not only for $W_\mu$ but for a large class of doubling
weights on the unit ball.

The new Bernstein inequalities on the ball are analogs of the inequalities established in \cite{GX} for the simplex.
The latter also contains a function like $\Phi_i$ in \eqref{eq:B1}, which provides improvement over the classical
Bernstein inequalities. In both cases, the associated function appears naturally in the Bernstein inequality derived
using plurisubharmonic functions, developed in \cite{Baran}, in which $\Phi_i$ and its alike on the simplex appear
naturally as the reciprocal of the Dini derivative of an extremal function. The connection, which provides an
interpretation for the new Bernstein inequalities, will be described in the next section.

One possible approach to prove the new inequalities, such as \eqref{eq:B1}, is to follow the proof in \cite{GX} by
using the highly localized kernels developed in \cite{PX, X23}, which requires technique and tedious estimates.
We decide to choose a different approach and derive the Bernstein inequalities on the unit ball from those
established on the simplex by making use of the close relation between analysis on the unit ball and on the
simplex (cf. \cite{X06} as well as \cite{DaiX, DX}). Since the mapping from the simplex to the ball involves squaring
each coordinate, we need $2^d$ maps, one to each quadrant of the ball, so that we can recover all polynomials
on the ball. As each map has a distinct Jacobian, the weight on the simplex for each quadrant of the ball has to
be different. As a result, it is not entirely obvious how to put the pieces together at first sight. Hence, the proof
remains non-trivial and is of interest in its own.

Another class of new Bernstein inequalities that we shall prove is the sharp inequalities in the $L^2$ norm on the ball.
Such inequalities have been studied recently in \cite{K2023}, where it was shown that the sharp Bernstein inequality for
polynomials of odd degree is different, and slightly better, from the one for polynomials of even degree. The proof of
these inequalities can be carried out through the spectral operator, which is a second-order differential operator that has
associated orthogonal polynomials as eigenfunctions. The key step is to write the spectral operator in a self-adjoint form,
which is known for the classical weight function on the simplex and on the ball. It turns out that the self-adjoint form of the
spectral operator is not unique, which was first realized in \cite{GX} for the operator on the simplex, and we shall give a
new one for the operator on the unit ball, which leads to a new family of Bernstein inequalities on the ball that are
also sharp.

The paper is organized as follows. The next section is the preliminary, in which we recall basic results about orthogonal
polynomials on the unit ball and simplex, recall the Bernstein inequality on the simplex, and discuss their connection with
the extremal function associated with plurisubharmonic functions. Section three is devoted to establishing the weighted
$L^p$ Bernstein inequalities. Finally, the sharp $L^2$ inequalities are stated and proved in Section 4.

Throughout this paper, we denote by $c$, or $c'$ etc, a positive constant that depends only on fixed parameters,
whose value may change from line to line.

\section{Preliminary}
\setcounter{equation}{0}

We recall background and essential results on the unit ball and on the simplex. The first subsection is devoted to
orthogonal polynomials on the unit ball, and the second subsection discusses the relation between orthogonal polynomials
on the ball and on the simplex. The third subsection reviews the Bernstein inequalities on the simplex established recently
in \cite{GX}, and discusses their connection with an extremal function that is related to the Bernstein inequalities established
via plurisubharmonic functions.

\subsection{Orthogonal polynomials on the unit ball}
Let $\BB^d = \{x \in \RR^d: \|x\| \le 1\}$ be the unit ball in the $d$-dimension Euclidean space $\RR^d$, where
$\|\cdot\|$ denote the Euclidean norm of $x \in \RR^d$. The classical weight function $W_{\mu}$ on the unit ball
$\mathbb{B}^{d}$ is defined by
$$
   W_{\mu}(x)=\left(1-\|x\|^{2}\right)^{\mu}, \quad \mu> -1.
$$
The classical orthogonal polynomials on $\BB^d$ are orthogonal with respect to the inner product $\la \cdot, \cdot \ra_\mu$
of $L^2(W_\mu, \BB^d)$, defined by
$$
   \la f,g \ra_\mu =  \int_{\BB^d} f(x) g(x) W_\mu(x) \d x.
$$
Let $\mathcal{V}_{n}(W_{\mu}, \BB^d)$ denote the space of orthogonal polynomials of degree $n$ with respect
to the weight function $W_{\mu}$ on $\mathbb{B}^{d}$. It is well-known that $\dim \CV_n(W_\mu,\BB^d) = \binom{n+d-1}{n}$.
The space $\CV_n(W_{\mu}, \BB^d)$ has several orthogonal bases that can be given explicitly.

Parametrizing the integral over $\BB^d$ in Cartesian coordinates, an explicit basis can be given in terms of the
Gegenbauer polynomials, denoted by $C_n^\l$, which are polynomials orthogonal with respect to $(1-t^2)^{\l-\f12}$
on $[-1,1]$ for $\l > -\f12$. More precisely, associated with $x = (x_1, \ldots, x_d) \in \mathbb{R}^d$, define by $\xb_j$
a truncation of $x$, namely
\[
\mathbf{x}_0 = 0, \qquad
\mathbf{x}_j = (x_1, \ldots, x_j), \qquad 1 \le j \le d.
\]
Note that \( \mathbf{x}_d = x \). Associated with \( \alpha = (\alpha_1, \ldots, \alpha_d) \), define
\[
\alpha^j := (\alpha_j, \ldots, \alpha_d), \qquad 1 \le j \le d,
\qquad \text{and} \qquad \alpha^{d+1} := 0.
\]
For $\alpha \in \mathbb{N}_0^d$, let $|\a| = \a_1+\ldots + \a_d$, and define the polynomials $P_\alpha$ by
\begin{align}\label{secondb}
P_\alpha(W_\mu; x)  = \prod_{j=1}^d (1 - \| \mathbf{x}_{j-1} \|^2)^{\alpha_j / 2}  C_{\alpha_j}^{\lambda_j}\!\bigg(
      \frac{x_j}{\sqrt{1 - \| \mathbf{x}_{j-1} \|^2}} \bigg),
\end{align}
where $\lambda_j = \mu + |\alpha^{j+1}| + \frac{d - j+1}{2}$. Then $\{P_\a(W_\mu): |\a| = n\}$ is an orthogonal basis of
$\mathcal{V}_n(W_\mu, \BB^d)$ \cite[Proposition 5.2.2]{DX}.

Using the spherical-polar coordinates to parametrize the integral over $\BB^d$, another orthogonal basis of
$\CV_n(W_\mu, \BB^d)$ can be given in terms of the Jacobi polynomials, $P_n^{(\a,\b)}$, which are polynomials
orthogonal with respect to $(1-x)^\a(1+x)^\b$ on $[-1,1]$, and the spherical harmonics. The latter are the restrictions
of homogeneous harmonic polynomials on the unit sphere $\mathbb{S}^{d-1}$, and they are OPs on the unit sphere
$\mathbb{S}^{d-1}$ of $\RR^d$. Let $\mathcal{H}_{n}^{d}$ be the space of spherical harmonics of degree $n$ of $d$
variables. It is well-known that $\dim \CH_n^d = \binom{n+d-1}{n} - \binom{n+d-3}{n-2}$. Let
$\left\{Y_{\ell}^{n-2 m}: 1 \leq \ell \leq \operatorname{dim} \mathcal{H}_{n-2 m}\right\}$ be an orthonormal basis
of $\mathcal{H}_{n-2 m}^{d}$ for $0 \leq m \leq n / 2$. Define
\begin{equation}\label{firstb}
 Q_{\ell, m}^{n}\left(W_{\mu} ; x\right)=P_{m}^{\left(\mu, n-2 m+\frac{d-2}{2}\right)}\left(2\| x \|^{2}-1\right) Y_{\ell}^{n-2 m}(x).
\end{equation}
Then the set $\left\{Q_{\ell, m}^{n}\left(W_{\mu}\right): 0 \leq m \leq n / 2,1 \leq \ell \leq
\operatorname{dim} \mathcal{H}_{n-2 m}^d\right\}$
is an orthogonal basis of $\mathcal{V}_{n}(W_{\mu},\BB^d)$ (cf. \cite[(5.2.4)]{DX}).

Let $\partial_i$ denote the partial derivative in the $i$th variable and $\Delta$ the Laplacian operator
$\Delta:= \partial^2_1 + \cdots + \partial^2_d$. The restriction of $\Delta$ on the unit sphere is the
Laplace-Beltrami operator, denoted by $\Delta_0$, which has spherical harmonics as eigenfunctions.
More precisely \cite[(1.8.3)]{DaiX}, for $n =0,1,2,\ldots,$
\begin{align}\label{eq:LBe}
  \Delta_0 Y = -n (n + d - 2) Y,  \qquad \forall\, Y \in \mathcal{H}_n^d.
\end{align}
An analog of this profound property holds for classical orthogonal polynomials on the unit ball. More precisely,
we have \cite[(5.2.3)]{DX}
\begin{equation} \label{eq:eigenB}
  \CD_\mu P = - n (n+ 2 \mu + d) P, \qquad \forall P \in \CV_n(W_\mu, \BB^d),
\end{equation}
where $\CD_\mu$ is the second-order differential operator defined by
\begin{equation*}
  \CD_\mu := \sum_{i=1}^d (1-x_i^2) \partial_i^2 - 2 \sum_{1 \le i < j \le d} x_i x_j \partial_i \partial_j - (d+2\mu+1)
  \sum_{i=1}^d x_i \partial_i.
\end{equation*}
{
The multivariate Chebyshev operator $\CD_\mu$ for $\mu =-1/2$ was introduced in \cite{Gan1} (see also \cite{Gan2}).} This differential operator can be written in several different forms. We state one more below that can be used
to show that $\CD_\mu$ is self-adjoint in $L^2(W_\mu, \BB^d)$; see Section 4. We need to introduce
the differential operator $D_{i,j}$ defined by
$$
  D_{i,j}: = x_i \partial_j - x_j \partial_i, \qquad 1 \le i < j \le d.
$$
The differential operator $\CD_\mu$ can be decomposed as a sum \cite[Proposition 7.1]{DaiX1}
\begin{equation} \label{eq:CDdecomp}
  \CD_{\mu} = \CD_{\mu}^{\mathrm{rot}} + \CD_\SS,
\end{equation}
where $\CD_\mu^{\mathrm{rot}}$ and $\CD_\SS$ are the radius and the spherical parts, respectively, defined by
\begin{align}\label{DmuDec1}
  \CD_{\mu}^{\mathrm{rot}} : =  \frac{1}{W_\mu(x)}\sum_{i=1}^{d} \partial_{i}\,\!\big(W_{\mu+1}(x)\,\partial_{i}\big)
 \quad \hbox{and}\quad   \CD_\SS:= \sum_{1\le i < j \le d} D_{i,j}^2.
\end{align}
The reason we call $\CD_\mu^{\mathrm{rot}}$ will become clear in the last section. We call $\CD_\SS$ the
sphiercal part because $D_{i,j}$ is the angular derivative  \cite[(1.8.1)]{DaiX} in the sence that
if $(x_i,x_j) = r_{i,j}(\cos \t_{i,j}, \sin \t_{i,j})$,
then $D_{i,j} = \frac{\partial}{\partial \t_{i,j}}$. Moreover, the restriction of $\CD_\SS$ on the unit sphere
$\sph$ agrees with the Laplace-Beltrami operator \cite[(1.8.3)]{DaiX}.

Let us also mention that an orthogonal polynomial $P \in \CV_n(W_\mu; \BB^d)$ satisfies a parity property:
If $n$ is even, $P$ is a sum of monomials of even degree, and if $n$ is odd, then $P$ is a sum of monomials of
odd degree \cite[Theorem 3.3.11]{DX}. Finally, the Fourier orthogonal expansion for $f \in L^2(W_\mu, \BB^d)$ is
defined by
\begin{equation} \label{eq:Fourier}
f= \sum_{n=0}^\infty  \proj_n(W_\mu;f),
\end{equation}
where $\proj_n(W_\mu): L^2(W_\mu, \BB^d) \mapsto  \CV_n(W_\mu; \BB^d)$ is the projection operator and,
in terms the orthogonal basis $\{Q_{\ell,m}^n\}$ of $\CV_n(W_\mu; \BB^d)$, this operator can be written as
\begin{equation}\label{eq:proj}
\proj_n (W_\mu; f) = \sum_{m =0}^{\lfloor \frac n 2 \rfloor} \sum_{\ell=1}^{\dim \CH_{n-2m}^d}
     \wh f_{\ell, m}^n Q_{\ell,m}^n, \quad \hbox{where} \quad
         \wh f_{\ell, m}^n = \frac{\la f,  Q_{\ell,m}^n \ra_\mu}{\la  Q_{\ell,m}^n,  Q_{\ell,m}^n \ra}.
\end{equation}
By its definition, the projection operator is independent of the choice of bases of $ \CV_n(W_\mu; \BB^d)$.
For further discussions, see \cite[Section 5.2]{DX}.

\subsection{Orthogonal polynomials on the simplex} On the simplex $\triangle^d$ defined by
$$
\triangle^d := \left\{x \in \RR^d: x_1 \ge 0, \ldots, x_d \ge 0, \, |x| \le 1 \right\},
$$
where $|x| = x_1+\cdots+x_d$, the Jacobi weight is defined, for $\bk =(\k_1, \ldots, \k_{d+1})$, by
$$
  \sW_\bk(x) = x_1^{\k_1} \cdots x_d^{\k_d} (1-|x|)^{\k_{d+1}}, \qquad \k_i > -1, \quad 1 \le i \le d+1.
$$
Let $\CV_n(\sW_\bk, \triangle^d)$ be the space of orthogonal polynomials of degree $n$ for $\sW_\bk$ on the simplex.
An orthogonal basis for this space can be given in terms of the Jacobi polynomials \cite[Section 5.3]{DX}. Although
the explicit formulas of the basis will not be utilized in this study, the close relation between orthogonal polynomials on
the two domains is needed. Indeed, the space $\CV_n(W_\mu, \BB^d)$ on the unit ball can be decomposed as a direct
sum of the space $\CV_n(\sW_\bk, \triangle^d)$ on $\triangle^d$ of degree $m$ in $y$ with $y_j = x_j^2$, where $\bk$
takes different values depending on $\mu$. To emphasize the dependence on the variables, we denote by
$$
  \CV_n\left(\sW_\bk, \triangle^d\right)\circ\psi = \mathrm{span} \left\{P \circ \psi : P \in \CV_n(\sW_\bk, \triangle^d) \right\},
$$
where $\psi$ is defined by $\psi: x \in \RR_+^d \mapsto (x_1^2,\ldots, x_d^2)$. Then, the aforementioned statement is given
precisely as
\begin{align} \label{eq:VB=VT}
 \begin{split}
  \CV_{2n}\left(W_\mu; \BB^d\right) & =
     \bigoplus_{\ve \in \{0,1\}^d,\, |\ve| = \mathrm{even}} x^\ve \CV_{n-\f12 |\ve|} \left(\sW_{(-\f12 + \ve, \mu)},\triangle^d\right)\circ \psi, \\
  \CV_{2n+1}\left(W_\mu; \BB^d\right) & =
     \bigoplus_{\ve \in \{0,1\}^d,\, |\ve| = \mathrm{odd}} x^\ve \CV_{n-\f12 (|\ve|-1)} \left(\sW_{(-\f12 + \ve, \mu)}, \triangle^d\right)\circ \psi.
\end{split}
\end{align}
These relations can be deduced fairly straightforwardly from the following integral identity (cf. \cite[Lemma 4.4.1]{DX}),
\begin{align}\label{BtoT}
\int_{\mathbb{B}^{d}} f(y_1^2, \ldots, y_d^2)\, \mathrm{d}y
   = \int_{\triangle^d} f(x_1, \ldots, x_d)
     \frac{\mathrm{d}x}{\sqrt{x_1 \cdots x_d}}.
\end{align}

The orthogonal polynomials in $\CV_n\left(\sW_\bk, \triangle^d\right)$ are also the eigenfunctions of a second-order linear
differential operator, which plays an important role in establishing sharp Bernstein inequalities in $L^2(\sW_\bk, \triangle^d)$
in \cite{GX}. We shall not state the operator here.

The connection between the orthogonal structure on the ball and on the simplex also extends to some other aspects of the
analysis on the two domains, and will be used later in Section 4. We mention one such property below, as a preparation for
the discussion on the Bernstein inequalities in the following subsection.

Let $\Omega$ be a domain equipped with a distance function $\sd(\cdot,\cdot)$. For $r > 0$ and $x\in \Omega$, denote
by $\sB(x,r) = \{y \in\Omega: \sd(x,y) < r\}$ the ball centered in $x$ with radius $r$. A weight function $\sw$ defined
on $\Omega$ is called a doubling weight if there is a constant $L> 0$ such that
$$
    \sw (\sB(x, 2 r)) \le L\, \sw (\sB(x, 2 r)), \qquad \forall x\in \Omega, \quad 0 < r < r_0,
$$
where $r_0$ is the largest positive number such that $\sB(x, r) \subset \Omega$, and, for a set $E \subset \Omega$, we define
$\sw(E)$ by $\sw(E) = \int_E \sw(x) \d x$.

The distance $\sd_\triangle$ on the simplex $\triangle^d$ and the distance function $\sd_{\mathbb{B}}$ on the ball $\mathbb{B}^d$
are defined by, respectively,
\begin{align*}
  &\sd_\triangle(x,y) = \arccos \left( \sqrt{x_1} \sqrt{y_1} + \cdots  \sqrt{x_d} \sqrt{y_d}+  \sqrt{1-|x|} \sqrt{1-|y|}\right), \\
  &\sd_{\mathbb{B}} (x,y) = \arccos \left(\langle x, y\rangle+\sqrt{1-\|x\|^{2}} \sqrt{1-\|y\|^{2}}\right).
\end{align*}
They are clearly closely related, as $\sd_\triangle(x^2, y^2) = \sd_\BB(x,y)$ for $x,y$ in the same quadrant of $\BB^d$, if
we define $x^2 =(x_1^2,\ldots,x_d^2)$. For later reference, we state the following lemma, which follows immediately from
the relation between the distance functions of the two domains and the identity \eqref{BtoT}.

\begin{lem}\label{lemDW}
A weight function $W(x)=\sw\big(x_1^2,\ldots,x_d^2\big)$ is a doubling weight on $\BB^d$ if and only if the weight
function
$$
 \sW_\triangle(x) = \frac{\sw(x_1,\ldots,x_d)}{\sqrt{x_1 \cdots x_d}}
$$
is a doubling weight on the simplex $\triangle^d$.
\end{lem}

\subsection{Bernstein Inequalities on the Simplex}
Let $\Pi_n^d$ be the space of algebraic polynomials of degree at most $n$ in $d$ variables. For $1 \le p \le \infty$,
we denote by $\|f\|_{\bk,p}$ the norm of $f\in L^p(W_\bk,\triangle^d)$, where we assume that the norm is the uniform
norm on $\triangle^d$ when $p = \infty$. The Bernstein inequalities on the simplex bounded the norm of the
partial derivatives $\partial_i$ and the derivatives along other edges of the simplex,
$$
    \partial_{i,j}:= \partial_i - \partial_j, \qquad 1 \le i \ne j \le d.
$$
The classical Bernstein inequalities on the simplex are defined by, for $f \in \Pi_n^d$,
\begin{align*}
   \left \| \varphi_i^r \partial_i^r f \right \|_{\bk,p} & \le c_{p,r} n^r {\|f\|}_{\bk,p}, \qquad 1 \le i \le d,
\end{align*}
and
\begin{align*}   \left\| \varphi_{i,j}^r \partial_{i,j}^r f \right\|_{\bk,p} & \le c_{p,r} n^r {\|f\|}_{\bk,p}, \quad 1 \le i < j \le d,
\end{align*}
where the functions $\varphi_i$ and $\varphi_{i,j}$ are given by
$$
\varphi_i(x) = \sqrt{x_i}\sqrt{1-|x|}, \quad 1 \le i \le d, \quad \hbox{and}\quad \varphi_{i,j} = \sqrt{x_i}\sqrt{x_j}, \quad 1\le i<j \le d,
$$
and $c_{p,r}$ is a positive constant independent of of $n$ (cf. \cite{BX, Dit, DT}).

More recently, several new and stronger Bernstein inequalities were established in \cite{GX}, which hold not only for
$\|\cdot\|_{\bk, p}$ but also for the $L^p$ norm defined for any doubling weight on the simplex \cite[Theorem 3.1]{GX}.

\begin{thm}\label{thm:B-Lp}
Let $\sW$ be a doubling weight on $\triangle^d$ and define, for $x \in \triangle^d$,
$$
 \quad \phi_{i} (x) := \frac{\sqrt{x_i}\sqrt{1-|x|}}{\sqrt{x_i+ 1- |x|}} \quad \hbox{and} \quad
   \phi_{i,j} (x) := \frac{\sqrt{x_i} \sqrt{x_j}}{\sqrt{x_i+x_j}}.
$$
For $r \in \NN$ and $1 \le p \le \infty$, there exists a constant $c=c(\sW,r,d,p) > 0$  such that for every $f \in \Pi_n^d$, the following inequalities hold.
\begin{equation}\label{eq:B_p1}
   \left\| \phi_i^r \partial_i^r f \right\|_{L^p(\sW, \triangle^d)} \le c n^r {\|f\|}_{L^p(\sW, \triangle^d)},\quad 1 \le i \le d,
\end{equation}
and
\begin{equation}\label{eq:B_p2}
   \left\| \phi_{i,j}^r \partial_{i,j}^r f \right\|_{L^p(\sW, \triangle^d)} \le c n^r {\|f\|}_{L^p(\sW, \triangle^d)}, \quad 1 \le i < j \le d.
\end{equation}
\end{thm}

The two inequalities in the theorem are stronger than the classical Bernstein inequalities \eqref{eq:B_p1} and \eqref{eq:B_p1}
when $\sW = \sW_\bk$ is the Jacobi weight, since $\varphi_i(x) \le \phi_i(x) \le 1$ and $\varphi_{i,j}(x)\le \phi_{i,j}(x) \le 1$ for
$x \in \triangle^d$. 
While the factor $\sqrt{x_i}$ and $\sqrt{1-|x|}$ in $\varphi_i$ can be interpreted as the distance from $x$ to the boundary
of $\triangle^d$, it is not clear if the function $\phi_i$ has a natural geometric interpretation. It turns out, however, that
an interpretation of $\phi_i$ lies in the theory of an extremal function associated with the Bernstein inequality that
utilizes plurisubharmonic functions. This connection sheds light on the new Bernstein inequalities in Theorem \ref{thm:B-Lp}
from a different angle, which we now describe.

Let $E$ be a compact subset of $\mathbb{C}^d$. Denote the uniform norm on $E$ by $\|\cdot\|_E$.
The Siciak's extremal function on $E$, denoted by $\Phi_E(z)$, is defined for $z \in \mathbb{C}^d$ by
\begin{align}\label{extremPhi}
\Phi_{E}(z):=\sup \left\{|P(z)|^{\frac{1}{\operatorname{deg} P}}: \operatorname{deg} P \geq 1
\text { and } {\|P\|}_{E} \leq 1, \quad  P \in \CP\right\},
\end{align}
where $\CP$ denotes the space of holomorphic polynomials.
We refer to \cite{S1962} for properties of this function and its applications in the theory of analytic
functions in several complex variables. To state the result most relevant to us, we need the definition of {\it plurisubharmonic} (psh)
functions. A function $u$ with values in $[-\infty, +\infty)$ defined in an open set $X \subset \mathbb{C}^n$ is a psh function
if
\begin{enumerate}
    \item $u$ is upper semi-continuous;
    \item For arbitrary $z$ and $w$ in $\mathbb{C}^n$ the function
   \[
        \tau \mapsto u(z + \tau w)
  \]
is subharmonic in the open subset of $\mathbb{C}$ where it is defined.
\end{enumerate}
If both $u$ and $-u$ are plurisubharmonic, then $u$ is called \textit{pluriharmonic}.
The Lelong class of psh functions is defined by
$$
\mathcal{L}_{d} := \left \{ u \ \text{psh on } \mathbb{C}^d: \ u(z) \leq \log(1 + \sqrt{|z_1|^2 + \ldots + |z_d|^2}) + O(1) \right \}.
$$
One of the basic results for Siciak's extremal function is its connection with the function
\begin{equation} \label{extremE}
V_E(z):= \sup \{ u(z) : u \in \mathcal{L}_d, \, u|_E \le 0 \}.
\end{equation}
For $d =1$, the function $V_E$ is the classical Green function of the planar compact set $E$ that has a logarithmic
pole at infinity. The following theorem is due to Zakharyuta  \cite{Z} and Siciak \cite{S1981}.

\begin{thm}
If $E$ is a compact subset of $\mathbb{C}^d$ then
\[
\log \Phi_E(z) = V_E(z) \quad \text{for } z \in \mathbb{C}^d.
\]
\end{thm}

Let $f$ be real-valued in a neighborhood of $x_0 \in \mathbb{R}$. The lower Dini derivative $D_{+}f$, also called a
lower right-hand derivative, of $f$ at $x_0$ is defined by
\begin{align*}
 D_{+}f(x_0):=\liminf\limits_{h \rightarrow 0^{+}} \frac{f(x_0+h) - f(x_0)}{h}.
\end{align*}
Let $\{e_1,\ldots,e_d\}$ be the standard orthogonal basis in $\mathbb{R}^d$. Let $E$ be a compact set in $\mathbb{C}^d$.
For $z \in \operatorname{int} E$ and each $j = 1, \ldots, d$, define $F_j : \mathbb{R} \to \mathbb{R}$ by
\[
F_j(t):=  V_E\bigl(z + \i\, t e_j\bigr), \quad t \in \mathbb{R}.
\]
Since $\Phi_E(z)=1$ for $z \in E$, it follows $V_E(z)=0$ by $\log \Phi_E(z) = V_E(z)$. Therefore, for $z \in \operatorname{int} E$,
\[
D_{+}F_j(0)= \liminf\limits_{\epsilon \rightarrow 0^{+}} \frac{V_E(z + \i\, \epsilon e_j) - V_E(z)}{\epsilon}
   = \liminf\limits_{\epsilon \rightarrow 0^{+}} \frac{V_E(z + \i\,\epsilon e_j)}{\epsilon}=:D^{+}_j V_E(z).
\]
By its definition, $D^{+}_j V_E$ can be called a Dini derivative of the extremal function $V_E$. These derivatives are
closely related to the Bernstein-type inequalities, as seen in the following theorem \cite{B1, B2}, in which a compact
set $K \subset \RR^d$ is treated as a subset of $\mathbb{C}^{d}$ such that $\mathbb{R}^{d}=\{(z_{1}, \ldots, z_{d}) \in \mathbb{C}^{d}: \operatorname{Im} z_{j}=0, j=1, \ldots, d\}$.

\begin{thm}
Let $K$ be a compact set in $\mathbb{R}^d$ with nonempty interior. For every $x \in \operatorname{int}K$ and $f \in \Pi_n^d$,
\begin{align}\label{Baran_extremal}
  |\partial_j f(x)| \leq n D^{+}_j V_K(x) \left({\|f\|}^2_K - f^2(x)\right)^{1/2}, \quad j=1,\ldots,d.
\end{align}
\end{thm}

To see the connection to the Bernstein inequality on the simplex, we recall the explicit formula of $\Phi_K$ for
$K = \triangle^d$ given in \cite{Baran},
\begin{align*}
 \Phi_{\triangle^d}(z)=\left[ h(|z_1| + \ldots + |z_d| + |z_1 + \ldots + z_d -1|) \right]^{1/2},
\end{align*}
where $h(\zeta)=\zeta+\sqrt{\zeta^2-1}$ if we choose a branch of the square root function so that $|h(\zeta)|>1$ for
$\zeta \in \mathbb{C}\setminus [-1,1]$, from which one can compute the Dini derivative explictly,
\begin{align*}
      D^{+}_i V_{\triangle^d}(x)= \frac{\sqrt{x_i+ 1- |x|} }{\sqrt{x_i}\sqrt{1-|x|}} =  \frac{1} {\phi_{i} (x)}.
\end{align*}
This shows, in particular, that $\phi_i(x)$ in \eqref{eq:B_p1} appears as the reciprocal of the Dini derivative that
appears in the Bernstein inequality \eqref{Baran_extremal}. Furthermore, moving $ D^{+}_j V_K(x)$ to the left-hand
side of \eqref{Baran_extremal} shows that \eqref{eq:B_p1} can be regarded as an $L^p$ version of the inequality
\eqref{Baran_extremal} for $K= \triangle^d$. As far as we know, the $L^p$ version of the latter has only been
discussed for certain cuspidal domains in \cite{BT}.

\section{$L^p$ Bernstein Inequalities for Doubling Weight}\label{last}
\setcounter{equation}{0}

In this section, we prove our new Bernstein inequalities on the unit ball in the $L^p$ norm with respect to a fully symmetric
doubling weight. The main results are stated and discussed in the first subsection, while their proof is given in the second
subsection, 
{and two examples that show the sharpness of the inequalities are proved in the third subsection.}

\subsection{Main result}
For comparison, let us mention the following two Bernstein inequalities on $\BB^d$ that are known in
the literature (cf. \cite[(12.3.17)]{DaiX} and \cite{Dai}). For $1 \le p\le \infty$, $f \in \Pi_n^d$ and $r \in \NN$,
\begin{equation}\label{old_pi}
\left\|\varphi^r  \partial_{i}^r f\right\|_{L^p(W_\mu, \mathbb{B}^{d})} \leq c\, n^r {\|f\|}_{L^p(W_\mu, \mathbb{B}^{d})}, \quad 1 \leq i \leq d,
\end{equation}
where we recall $\varphi(x) = \sqrt{1-\|x\|^2}$, and
\begin{equation}\label{old_pij}
\left\|  D_{i, j}^r f \right\|_{L^p(W_\mu, \mathbb{B}^{d})} \leq c \, n^r {\|f\|}_{L^p(W_\mu, \mathbb{B}^{d})}, \quad 1 \leq i<j \leq d,
\end{equation}
where we recall that $D_{i,j} = x_i \partial_j - x_j \partial_i$ is the angular derivative.

To state our main result, we need two functions, $\Phi_i$ and $\Phi_{i,j}$, which play the role of $\phi_i$ and
$\phi_{i,j}$ in the Bernstein inequalities on the simplex discussed in Theorem \ref{thm:B-Lp}.
For $x = (x_1,\ldots, x_d) \in \mathbb{B}^{d}$ and $1 \le i, j \le d$, they are defined by
\[
 \Phi_{i}(x):=\frac{ \sqrt{1-\|x\|^2}}{\sqrt{x^2_{i}+1-\|x\|^2}}  \quad \text { and } \quad \Phi_{i, j}(x):=\frac{1}{\sqrt{x^2_{i}+x^2_{j}}}.
\]
Our main result for the Bernstein inequality on the ball is the following theorem.

\begin{thm}\label{BIforLp}
{
Let $d>1$ and let} $W(x)=\sW(x^2_1,\ldots,x^2_d)$ be a doubling weight on $\mathbb{B}^{d}$. For $1 \le p \le \infty$, $r \in \NN$,
and $f \in \Pi_n^d$,
\begin{equation}\label{pi}
\left\|\Phi_{i}^r \partial_{i}^r f\right\|_{L^p(W, \mathbb{B}^{d})} \leq c\, n^r \|f\|_{L^p(W, \mathbb{B}^{d})}, \quad 1 \leq i \leq d,
\end{equation}
where $r$ is a positive integer, and
\begin{equation}\label{pij}
\left\|\Phi_{i, j} D_{i, j} f \right\|_{L^p(W, \mathbb{B}^{d})} \leq c \, n \|f\|_{L^p(W, \mathbb{B}^{d})}, \quad 1 \leq i<j \leq d,
\end{equation}
where $c$ is a positive constant independent of $n$.
\end{thm}

The inequalities \eqref{pi} and \eqref{pij} are stronger than \eqref{old_pi} and \eqref{old_pij} since $\Phi_i(x) \ge \varphi(x)$
and $\Phi_{i,j}(x) \ge 
1$ for all $1\le i,j \le d$. 
{Furthermore, they are stronger in the order of magnitude for
certain polynomials.
\begin{exam}\label{exam1}
For $x = (x_1,\ldots, x_d) \in \BB^d$, let $f_{i,k}^n(x) = \frac{1}{r!} x_k^r x_i^{n} $. Then, for $1\le k \ne i \le d$,
\begin{equation}\label{eq:exam1}
 \frac{\left\|\Phi_{k}^r \partial_{k}^r f_{i,k}^n \right\|_{L^p(W_\mu, \mathbb{B}^{d})}}{ \left\|\varphi^r  \partial_{k}^r f_{i,k}^n \right\|_{L^p(W_\mu, \mathbb{B}^{d})} }
  = \left(\frac{\Gamma(\mu + \f32)\Gamma(\frac{np+d}2 +\mu + 1 + \frac{r p }{2})}{\Gamma(\frac{rp}{2} + \mu + \f32)\Gamma(\frac{np+d}2+\mu + 1)} \right)^{\f1p} \sim n^{\frac{r}{2}}
\end{equation}
In particular, this shows that the inequality \eqref{pi} is stronger than \eqref{old_pi} by an order of magnitude of $n^{\frac{r}{2}}$ for
the monomials $f_{i}^n$.
\end{exam}
\begin{exam}\label{exam2}
Let $d \ge 3$, $1 \le i < j \le d$, and $1\le k \le d$ such that $k \ne i$ and $k \ne j$. For $x = (x_1,\ldots, x_d) \in \BB^d$, let
$f_{i, k}^n(x) =x_i x_k^n$. Then
\begin{equation}\label{eq:exam2}
 \frac{\left\|\Phi_{i, j} D_{i, j} f_{i,k}^n \right\|_{L^p(W_\mu, \mathbb{B}^{d})}}{\left\|  D_{i, j} f_{i,k}^n \right\|_{L^p(W_\mu, \mathbb{B}^{d})}}
  = \left( \frac{ \Gamma(\frac{np+d}2+\mu +1+ \f{p}{2})}
    {\Gamma(\frac{p+2}{2}) \Gamma( \frac{np+d}2 + \mu+1)} \right)^{\f1p} \sim n^{\frac{1}{2}}.
\end{equation}
In particualr, this shows that the inequality \eqref{pij} is stronger than \eqref{old_pij} by an order of magnitude of $\sqrt{n}$ for
the monomials $f_{i, k}^n$.
\end{exam}
}

A couple of further remarks on the inequalities \eqref{pi} and \eqref{pij} are in order.
\begin{rem}
The denominators in $\Phi_i$ and $\Phi_{i,j}$ do not reduce a singularity for the integral. For $\Phi_i$, this is evident since
$0 \le \Phi_i (x) \le 1$. For $\Phi_{i,j}$, this follows from the definition of $D_{i,j}$, which shows
$$
    \Phi_{i,j}(x) D_{i,j} = \frac{x_i}{\sqrt{x^2_i+x^2_j}} \partial_j -   \frac{x_j}{\sqrt{x^2_i+x^2_j} }\partial_i,
$$
so that both factors in front of the derivatives have values in $[0,1]$
\end{rem}

\begin{rem}
The inequality \eqref{pij} for $D_{i,j}$ does not hold for $D_{i,j}^r$ with $r > 1$ in general. Indeed, a quick computation shows,
for example,
$$
  D_{i,j}^2 = x_i^2 \partial_i^2 + x_j^2 \partial_j^2 - x_i \partial_j - x_j \partial_i,
$$
where $\Phi_{i,j}^2(x) = \frac{1}{x_i^2+x_j^2}$, so that the first order partial derives in $\frac{1}{\Phi_{i,j}^2} D_{i,j}^2$
has a singularity of the first order. Furthermore, the inequality also does not hold for $(\Phi_{i,j} D_{i,j})^r$ for $r >1$,
since $(\Phi_{i,j} D_{i,j})^r = \Phi_{i,j}^r D_{i,j}^r$ as can be seen from $D_{i,j} \Phi_{i,j}(x) =0$. We note, however,
that if $W^{p(r-1)} \Phi_{i,j}$ is a doubling weight for $p \ge 1$ and $r>1$, then the inequality
$$
\left\| \Phi_{i,j}^2 D_{i,j}^2 f \right \|_{L^p(W, \BB^d)} \le c\, n^2 \left \| f \right \|_{L^p(W, \BB^d)}
$$
holds, as can be seen by following the proof of \eqref{pi} for $r > 1$. The condition, however, does not hold for
the classical weight function $W_\mu$ if $r \ge 2$.
\end{rem}

\begin{rem}
The above inequalities are related to Siciak's extremal function, in a way similar to the case of the simplex, as
discussed in Subsection 2.3. The explicit formula for the extremal function $\Phi_E$ in \eqref{extremPhi}
for $E = \BB^d$ is given in \cite{Baran},
\begin{align*}
  \Phi_{\mathbb{B}^{d}}(z)=\left(h(|z_1|^2 + \ldots + |z_d|^2 + |z^2-1|) \right)^{1/2},
\end{align*}
form which one can deduce that the Dini derivative of $V_{\BB^d}$, defined in \eqref{extremE}, is
\begin{align*}
  D^{+}_i V_{\mathbb{B}^{d}}(x)= \frac{\sqrt{x^2_{i}+1-\|x\|^2}}{\sqrt{1-\|x\|^2}}= \frac{1} { \Phi_{i}(x)}.
\end{align*}
Thus, just like $\phi_i$ for the simplex, $\Phi_i$ is the reciprocal of the Dini derivative of $V_E$ for $E = \BB^d$.
In particular, this further enforces the suggestion that the inequality \eqref{pi} can be seen as the $L^p$ version of the inequality \eqref{Baran_extremal}
when $K = \mathbb{B}^d$.
\end{rem}

\subsection{Proof of Theorem~\ref{BIforLp}}

We introduce the following notation. For any function $f : \mathbb{R}^d \rightarrow \mathbb{R}$ and $\ve \in  \{0,1\}^d$, define
$$
   f_\ve (x) = \prod_{i=1}^d x_i^{\ve_i}  \frac{1}{2^d}  \sum_{ \tau \in \{1, -1\}^d } f (\tau x),
$$
where $\tau x = (\tau_1 x_1, \ldots, \tau_d x_d)$. Then
$$
      f(x) = \sum_{\ve \in \{0,1\}^d} f_\ve (x).
$$
If $\ve_i =0$, $f_\ve$ is even in $x_i$ and if $\ve_i = 1$, $f_\ve$ is odd in $x_i$. For each
$\ve$, define the index set $J(\ve) = \{j: \ve_j =1\}$, and let $x_\ve =  \prod_{j \in J(\ve)} x_j$. Then, if $f$ is a polynomial of degree $n$,
the parity of $f_\ve$ means that we can write $f_\ve$ as
$$
  f_\ve(x) = x_\ve \cdot g_\ve \big(x_1^2, \ldots, x_d^2\big),
$$
where $g$ is a polynomial of degree $(n - |J(\ve)|)/2$. This construction is motivated by the relation between orthogonal
polynomials on the unit ball and on the simplex, as shown in \eqref{eq:VB=VT}, which shows, in particular, how polynomials
on the unit ball can be generated by polynomials on the triangle by using $x \mapsto (x_1^2,\ldots,x_d^2)$ and $x_\ve$.

To illustrate the above notation and clarify its meaning, let us consider the case $d = 2$. For a function $f : \mathbb{R}^2 \to \mathbb{R}$, we have
$\ve \in \{0,1\}^2 = \{(0,0), (1,0), (0,1), (1,1)\}$.
Then
\[
   f_{\ve}(x_1, x_2)
   =  x_1^{\ve_1} x_2^{\ve_2} \frac{1}{4} \sum_{\tau_1, \tau_2 \in \{1, -1\}}  f(\tau_1 x_1, \tau_2 x_2).
\]
Hence,
\[
   f(x_1, x_2)
   = f_{00}(x_1, x_2)
     + f_{10}(x_1, x_2)
     + f_{01}(x_1, x_2)
     + f_{11}(x_1, x_2),
\]
where
\begin{itemize}
  \item $f_{00}$ is even in both variables,
  \item $f_{10}$ is odd in $x_1$ and even in $x_2$,
  \item $f_{01}$ is even in $x_1$ and odd in $x_2$,
  \item $f_{11}$ is odd in both variables.
\end{itemize}
If $f$ is a polynomial of degree $n$, then the parity of each component implies that
\[
\begin{aligned}
   f_{00}(x_1, x_2) &= g_{00}(x_1^2, x_2^2), \\
   f_{10}(x_1, x_2) &= x_1\, g_{10}(x_1^2, x_2^2), \\
   f_{01}(x_1, x_2) &= x_2\, g_{01}(x_1^2, x_2^2), \\
   f_{11}(x_1, x_2) &= x_1 x_2\, g_{11}(x_1^2, x_2^2),
\end{aligned}
\]
where each $g_{\ve}$ is a polynomial of degree $(n - |J(\ve)|)/2$.

With the above notation, any polynomial $f \in \Pi_n^d$ can be represented in the following form:
 \[
 f(x) =\sum_{\ve \in \{0,1\}^d} f_{\ve}(x)=\sum_{\ve \in \{0,1\}^d} x_\ve \cdot g_\ve \big(x_1^2, \ldots, x_d^2\big),
 \]
 where each $g_\ve$ is a polynomial of degree less than or equal to $n/2$. Then, by the triangle inequality and the linearity of $\partial_i$ and $D_{i,j}$,
\begin{align}
  &\left\|\Phi_{i} \partial_{i} f\right\|_{L^p(W, \mathbb{B}^{d})} \leq \sum_{\ve \in \{0,1\}^d} \left\|\Phi_{i} \partial_{i} f_\ve \right\|_{L^p(W, \mathbb{B}^{d})}, \label{inq:first} \\
  &\left\|\Phi_{i, j} D_{i, j} f \right\|_{L^p(W, \mathbb{B}^{d})} \leq \sum_{\ve \in \{0,1\}^d}  \left\|\Phi_{i, j} D_{i, j} f_\ve \right\|_{L^p(W, \mathbb{B}^{d})}. \label{inq:second}
\end{align}
Since $W$ is a reflection-invariant weight function (i.e. $W(x)=W(|x_1|,\ldots,|x_d|)$), we have
\begin{align*}
\left\| f_\ve \right\|_{L^p(W, \mathbb{B}^{d})} = 2^d \left\| f_\ve \right\|_{L^p(W, \mathbb{B}^{d}_{+})},
\end{align*}
where $\mathbb{B}^{d}_{+}:=\{x \in \mathbb{B}^{d} : x_i \geq 0, \, i=1,\ldots,d\}$. On the other hand,
\begin{align*}
\left\| f_\ve \right\|^p_{L^p(W,\mathbb{B}^{d}_{+})}&= \int_{\mathbb{B}^{d}_{+}} \left|  \prod_{i=1}^d x_i^{\ve_i}   \frac{1}{2^d}
   \sum_{ \tau \in \{1, -1\}^d } f (\tau x) \right|^p W(x) \d x
\\ &\leq \frac{1}{2^{d}} \sum_{ \tau \in \{1, -1\}^d }  \int_{\mathbb{B}^{d}_{+}} \left|  f (\tau x) \right|^p W(x) \d x =
 \int_{\mathbb{B}^{d}} \left|  f (x) \right|^p W(x) \d x = \left\| f \right\|^p_{L^p(W,\mathbb{B}^d)}.
\end{align*}
Hence, for every $\ve \in \{0,1\}^d$, we have
\begin{align}\label{BtoBplus}
\left\| f_\ve \right\|_{L^p(W, \mathbb{B}^{d}_{+})} \leq  \left\| f \right\|_{L^p(W, \mathbb{B}^{d})}.
\end{align}
Thus, by \eqref{inq:first} and \eqref{inq:second}, the proof of the main result for $r =1$ is reduced to establish the Bernstein
inequalities for $f_\ve$ for each $\ve \in \{0, 1\}^d$.

For the partial derivatives $\partial_i f_\ve$, the analysis can essentially be reduced to two cases: $i \notin J(\ve)$ and $i \in J(\ve)$.
First, consider the case $i \notin J(\ve)$. This means that the variable $x_i $ does not appear in the expression $x_\ve$.
In this situation, we have
\[
\partial_i  f_\ve = \partial_i \left\{ x_\ve \cdot g_\ve \big(x_1^2, \ldots, x_d^2\big) \right\} = x_\ve  2x_i \partial_i g_\ve \big(x_1^2, \ldots, x_d^2\big).
\]
Therefore, by the symmetry of the integrand, we have
\begin{align}\label{symmetry}
\left\|\Phi_{i} \partial_{i} f_\ve \right\|^p_{L^p(W, \mathbb{B}^{d})} = 2^d \int_{\mathbb{B}^{d}_{+}} \left| \Phi_{i}(x) 2x_i x_\ve  \partial_i g_\ve \big(x_1^2, \ldots, x_d^2\big) \right|^p W(x) \d x.
\end{align}
Then, by \eqref{BtoT} and the identity $\phi_{i}(x^2_1,\ldots,x_d^2)=x_i \Phi_{i}(x_1,\ldots,x_d)$, we obtain
\begin{align*}
\int_{\mathbb{B}^{d}_{+}} & \left| \Phi_{i}(x) 2x_i x_\ve  \partial_i g_\ve \big(x_1^2, \ldots, x_d^2\big) \right|^p W(x) \d x  \\
  & \qquad\qquad = 2^p \int_{\triangle^d} \left| \phi_{i}(u) \sqrt{u_\ve}  \partial_i g_\ve (u) \right|^p \sW(u) \frac{\d u}{ 2^d \prod_{l=1}^d \sqrt{u_l}}.
\end{align*}
Here $\sqrt{u_\ve} =  \prod_{j \in J(\ve)} \sqrt{u_j}$. Now, by Lemma \ref{lemDW}, we can apply the Bernstein inequality \eqref{eq:B_p1}
on the simplex to $g_\ve(u_1,\dots,u_d)$ with the doubling weight
\[
\Theta_{\ve}(u_1,\ldots,u_d)=   \frac{(\sqrt{u_\ve})^p \cdot \sW(u)}{ 2^d \prod_{l=1}^d \sqrt{u_l}},
\]
which leads to, after performing the change of variables $u_l=x^2_l$,
\[
\int_{\mathbb{B}^{d}_{+}} \left| \Phi_{i}(x) 2x_i x_\ve  \partial_i g_\ve \big(x_1^2, \ldots, x_d^2\big) \right|^p W(x) \d x \le c\, n^p \int_{\mathbb{B}^{d}_{+}} \left| f_\ve(x) \right|^p  W(x) \d x.
\]
Thus, by \eqref{symmetry}, we obtain
\[
\left\|\Phi_{i} \partial_{i} f_\ve \right\|_{L^p(W, \mathbb{B}^{d})} \leq  c\, n \|f_\ve\|_{L^p(W, \mathbb{B}^{d}_{+})} .
\]

Next, we consider the case in which $i$ belongs to $J(\ve)$. Taking the derivative,
\begin{equation} \label{eq:2terms}
\partial_i f_\ve=\partial_i \left\{ x_\ve \cdot g_\ve \big(x_1^2, \ldots, x_d^2\big) \right\}=   \frac{x_\ve}{x_i} g_\ve(x_1^2, \ldots, x_d^2) + 2x_i x_\ve\partial_i g_\ve(x_1^2, \ldots, x_d^2).
\end{equation}
We need to estimate the two terms on the right-hand side separately. The second term contains the derivative $\partial_i$ and it
has already been estimated above. It remains to estimate the first term, the one without the derivative. To this end, we will need
the following lemma \cite[Lemma 3.9]{GX}.
 \footnote{In \cite{GX}, the $\frac{\delta}{n^2}$ in the definition $\triangle_{n,\delta}^d$ is mistakenly written as $\frac{\delta}{n}$.}
\begin{lem} \label{lem:shrink}
Let $\sW$ be a doubling weight function on $\triangle^d$. For $\delta > 0$ and $n\in \NN,$ let
$$
\triangle_{n,\delta}^d =\{x\in\triangle^d: \tfrac\delta{n^2}< x_i \leq 1-\tfrac{\delta}{n^2}, 1 \le i \le d+1\},
$$
where $x_{d+1} = 1-|x|$. Then, for $f\in\Pi^d_n$, $1\leq p< \infty$, 
\begin{align}\label{shrink}
   \int_{\triangle^d}|f(x)|^p \sW(x)\d x\leq c_{\delta}\int_{\triangle_{n,\delta}^d}|f(x)|^p \sW(x)\d x.
\end{align}
\end{lem}
Applying (\ref{shrink}) to $g_\ve(u_1,\dots,u_d)$ with the weight
$
\Theta_{\ve}/ \sqrt{u^p_i}
$
gives
\begin{align*}
\int_{\triangle^d} \left| \frac{\sqrt{u_\ve}}{\sqrt{u_i}}  g_\ve(u)\right|^p \frac{ \sW(u)}{ 2^d \prod_{l=1}^d \sqrt{u_l}} \d u \leq c\, n^p \int_{\triangle^d} \left| \sqrt{u_\ve} g_\ve(u)\right|^p \frac{ \sW(u)}{ 2^d \prod_{l=1}^d \sqrt{u_l}} \d u.
\end{align*}
Making the substitution $u_l = x_l^2$ to go back to $\BB^d$ again, it follows that
\begin{align}\label{ShrinkBall}
\int_{\mathbb{B}^{d}_{+}} \left| \frac{x_\ve}{x_i} g_\ve(x_1^2, \ldots, x_d^2)\right|^p W(x) \d x \leq c\, n^p \int_{\mathbb{B}^{d}_{+}} \left|f_\ve(x) \right|^p W(x) \d x.
\end{align}
Therefore, by $\Phi_{i}(x) \leq 1$ and the symmetry of the integrand,
\[
\int_{\mathbb{B}^{d}} \left|\Phi_{i}(x) \frac{x_\ve}{x_i} g_\ve(x_1^2, \ldots, x_d^2)\right|^p W(x) \d x \leq 2^d c\, n^p \int_{\mathbb{B}^{d}_{+}} \left| f_\ve(x) \right|^p W(x) \d x,
\]
which takes care of the first term in the right-hand side of \eqref{eq:2terms}. Consequently, we obtain
in the case $i \in J(\ve)$,
 \[
\int_{\mathbb{B}^{d}} \left| \Phi_{i}(x) \partial_i f_\ve(x)\right|^p W(x) \d x \leq 2^{d+p-1}(2^p c' + c) n^p \int_{\mathbb{B}^{d}_{+}} \left| f_\ve(x)\right|^p W(x) \d x.
\]
Thus, we have shown that, for all $\ve \in \{0,1\}^d$ and $i \in \{1,\dots,d\}$,
\begin{align*}
\left\|\Phi_{i} \partial_{i} f_\ve\right\|_{L^p(W, \mathbb{B}^{d})} \leq c\, n \|f_\ve\|_{L^p(W, \mathbb{B}^{d}_{+})},
\end{align*}
which proves, by \eqref{inq:first} and \eqref{BtoBplus}, the desired inequality \eqref{pi} for $r =1$.

The proof for $r >1$ follows from iteration, similar to the proof in the case of the simplex in \cite{GX}. Indeed, for fixed $i$ and $p$,
define $W^*(x) = \Phi_i^{(r-1)p}(x) W(x)$, which is a doubling weight. Hence,
\begin{align*}
  \left \| \Phi_i^r \partial_i^r f \right \|_{L^p(W, \BB^d)}  \, & =   \left \|\Phi_i\partial_i \partial_i^{r-1} f \right \|_{L^p(W^*, \BB^d)} \\
   & \le c\, n   \left \| \partial_i^{r-1} f \right \|_{L^p(W^*, \BB^d)} = c \, n  \left \|\Phi_i^{r-1} \partial_i^{r-1} f \right \|_{L^p(W, \BB^d)},
\end{align*}
which allows us to complete the proof by iteration.

To prove the similar result for $D_{i, j}$, we have to consider four possible cases, depending on whether the indices $i$ and
$j$ are elements of $J(\ve)$ or not. First, we examine the case $i,j \in J(\ve)$. Then
\begin{align*}
D_{i, j} f_\ve(x_1,\dots,x_d)=\, &D_{i, j} \left\{x_\ve g_\ve(x_1^2, \ldots, x_d^2) \right\}  \\= \, & x_j  \frac{x_\ve}{x_i} g_\ve(x_1^2, \ldots, x_d^2)  + x_j x_\ve 2x_i \partial_{i} g_\ve(x_1^2, \ldots, x_d^2) \\
   & - x_i \frac{x_\ve}{x_j} g_\ve(x_1^2, \ldots, x_d^2) - x_i x_\ve  2x_j \partial_{j} g_\ve(x_1^2, \ldots, x_d^2).
\end{align*}
From inequality (\ref{ShrinkBall}), using $ \Phi_{i,j}(x)x_s \leq 1$ for $s \in  \{i,j\}$ and the symmetry of the integrand, we obtain
\begin{align*}
\int_{\mathbb{B}^{d}} \left| \Phi_{i,j}(x)  \left(\frac{x_j}{x_i}- \frac{x_i}{x_j}\right) f_\ve(x)\right|^p W(x) \d x
    \leq 2^d c \, n^p \int_{\mathbb{B}^{d}_{+}} \left| f_\ve(x)\right|^p W(x) \d x.
\end{align*}
Now, by applying the inequality \eqref{eq:B_p2} to $g_\ve(u_1,\dots,u_d)$ with the weight $\Theta_{\ve}$,
and performing the change of variables $u_l=x^2_l$, we obtain
\[
\int_{\mathbb{B}^{d}_{+}} \left| \Phi_{i,j}(x) x_j x_\ve 2x_i  \partial_{i,j} g_\ve(x_1^2, \ldots, x_d^2)\right|^p W(x) \d x \leq (2 c\,n)^p \int_{\mathbb{B}^{d}_{+}} \left| f_\ve(x)\right|^p W(x) \d x.
\]
Since the integrands in the integrals on the left-hand side are symmetric,
the integrals over $\mathbb{B}^{d}_{+}$ can be replaced by those over $\mathbb{B}^{d}$, provided that the right-hand side is multiplied by $2^d$. Therefore, it follows that
\begin{align}\label{Dij}
\left\|\Phi_{i,j} D_{i, j} f_\ve\right\|_{L^p(W, \mathbb{B}^{d})} \leq c\, n \|f_\ve\|_{L^p(W, \mathbb{B}^{d}_{+})} .
\end{align}
It remains to show that the above estimate is valid in the cases where $i \notin J(\ve)$ or $j \notin J(\ve)$. In these cases, we have
\begin{align*}
D_{i, j} \left\{x_\ve g_\ve(x_1^2, \ldots, x_d^2) \right\}   =x_j x_\ve  2x_i \partial_{i} g_\ve(x_1^2, \ldots, x_d^2) - x_i x_\ve   2x_j \partial_{j} g_\ve(x_1^2, \ldots, x_d^2) + R(x),
\end{align*}
where
\begin{align*}
R(x)=                                                                                   \begin{cases}
\, \frac{x_j}{x_i} x_\ve g_\ve(x_1^2, \ldots, x_d^2), & \text{if } i\in J(\ve),\ j\notin J(\ve),\\[2pt]
\, -\frac{x_i}{x_j} x_\ve g_\ve(x_1^2, \ldots, x_d^2), & \text{if } j\in J(\ve),\ i\notin J(\ve),\\[2pt]
0, & \text{if } i,j\notin J(\ve).
\end{cases}
\end{align*}
The proofs for these cases can be derived directly from the case $i, j \in J(\ve)$. Thus, the inequality~\eqref{Dij} is valid for
every $\ve$. By symmetry, \eqref{inq:second} and \eqref{BtoBplus}, we then obtain inequality~\eqref{pij}. This concludes
the proof for $1\leq p< \infty$. The case $p=\infty$ proceeds analogously, and we omit the details.
 \qed

{\subsection{Proof for examples}
To verify Example \ref{exam1}, we use the integral identity
\begin{equation} \label{eq:int-beta}
 \int_{\BB^d} |x_k|^a (1-\|x\|^2)^b \d x = \pi^{\f{d-1}2} \frac{\Gamma(\frac{a+1}{2}) \Gamma(b+1)}{\Gamma(\frac{a+d}{2}+b+1)},
\end{equation}
where $a, b > -1$, $1 \le k \le d$, and $x = (x_1,\ldots, x_d)$, which can be easily verified, and it holds for the case $\BB^1 = [-1,1]$.
For convenience, we denote by $c_d^b$ the constant when $a = 0$; that is, using $\Gamma(\f12) = \sqrt{\pi}$,
\begin{equation} \label{eq:cdb}
   c_d^b =  \int_{\BB^d}  (1-\|x\|^2)^b \d x = \pi^{\f{d}2} \frac{\Gamma(b+1)}{\Gamma(\frac{d}{2}+b+1)}.
\end{equation}
An immediate consequence of \eqref{eq:int-beta} gives
$$
\left\|\varphi^r  \partial_{k}^r f_{i,k}^n \right\|_{L^p(W_\mu, \mathbb{B}^{d})}^p = \int_{\BB^d} |x_i|^{n p}(1-\|x\|^2)^{\mu + \frac{r p}{2}} \d x
  = \pi^{\f{d-1}2} \frac{\Gamma(\frac{np+1}{2}) \Gamma(\frac{rp}{2}+\mu+1)}{\Gamma(\frac{np+d}{2} +\frac{rp}{2} +\mu+1)}.
$$
Moreover, by symmetry, we can assume $i < d$ and only need to consider the integral
\begin{align*}
 \left\|\Phi_{d}^r \partial_{d}^r f_{i,d}^n \right\|_{L^p(W_\mu, \mathbb{B}^{d})}^p
   = \int_{\BB^d} \frac{ |x_i|^{n p}}{(x_d^2 + 1-\|x\|^2)^{\frac{rp}{2}}} (1-\|x\|^2)^{\mu + \frac{r p}{2}} \d x.
\end{align*}
Setting $x = (x', x_d)$ and making a change of varaible $x_d = \sqrt{1-\|x'\|^2} t$, we obtain
\begin{align*}
 \left\|\Phi_{d}^r \partial_{d}^r f_{i,d}^n \right\|_{L^p(W_\mu, \mathbb{B}^{d})}^p
& = \int_{\BB^{d-1}} |x_i|^{n p} \int_{-\sqrt{1-\|x'\|^2}}^{\sqrt{1-\|x'\|^2}}
       \frac{(1-\|x'\|^2- x_d^2 )^{\frac{rp}{2}+\mu}}{(1-\|x'\|^2)^{\mu + \frac{r p}{2}}} \d x_d \d x' \\
& = \int_{\BB^{d-1}} |x_i |^{n p} (1-\|x'\|^2)^{\mu+\f12} \d x' \int_{-1}^1
        (1-t^2)^{\mu+\frac{r p}{2}} \d t \\
& = \pi^{\frac{\d-1}2} \frac{\Gamma(\frac{rp}{2}  +\mu+1)\Gamma(\mu+\f32)\Gamma(\frac{n p +1}{2})}
    {\Gamma(\frac{rp+1}{2} +\mu+1) \Gamma(\frac{np+d}{2}  +\mu+1)}
\end{align*}
by \eqref{eq:int-beta} for $\BB^{d-1}$ and $\BB^1 = [-1,1]$. Putting these together verifies \eqref{eq:exam1}. The asymptotic as $n \to \infty$ follows from the well-known relation $\Gamma(z+a)/\Gamma(z+b) \sim z^{a-b}$ for $z \to \infty$. }

{
For Example \ref{exam2}, we need one more integral identity. For $d \ge r \ge 1$, write $x = (u_x, u_x')$
with $u_x = (x_1,\ldots, x_r)$. Then
\begin{equation} \label{eq:intB-B}
  \int_{\BB^d} f(u_x)(1-\|x\|^2)^\mu \d x = \int_{\BB^r} f(u)(1- \|u\|^2)^{\mu+ \f{d-r}{2}} \d u \int_{\BB^{d-r}} (1-\|v\|^2)^\mu \d v,
\end{equation}
which can be easily verified. By symmetry, we only need to verfiy \eqref{eq:exam2} for $i =1$, $j=2$, and $k =3$.
Since
$$
  D_{1,2} f_{1,3}(x) = D_{1,2} x_1 x_3^n = - x_2 x_3^n,
$$
applying \eqref{eq:intB-B} with $r =2$ and using the constant $c_{d-2}^\mu$ in \eqref{eq:cdb}, we obtain
\begin{align*}
  \left\|  D_{1, 2} f_{1,3}^n \right\|_{L^p(W_\mu, \mathbb{B}^{d})}^p & =
    \int_{\BB^d} |x_2|^p |x_3|^{n p} (1-\|x\|^2)^\mu \d x  \\
    & = c_{d-2}^\mu \int_{\BB^2}  |x_2|^p |x_3|^{n p} (1- x_2^2 - x_3^2)^{\mu + \frac{d-2}{2}}\d x_2 \d x_3 \\
    & = c_{d-2}^\mu \int_{-1}^1 |x_2|^p (1-x_2^2)^{\mu+\frac{d-1}{2} + \frac{np}{2}} \d x_2
    \int_{-1}^1 |t|^{n p} (1- t^2)^{\mu + \frac{d-2}{2}} \d t \\
    & = \pi^{\frac{d-2}{2}} \frac{\Gamma(\frac{np+1}{2}) \Gamma(\frac{p+1}{2})\Gamma(\mu+1) } {\Gamma(\frac{np+d}{2} +\mu +1+ \frac{p}{2})},
\end{align*}
where we have used \eqref{eq:int-beta} again. Now, applying \eqref{eq:intB-B} with $r =3$, we obtain
\begin{align*}
    \left\|\Phi_{1, 2} D_{1, 2} f_{1,3}^n \right\|_{L^p(W_\mu, \mathbb{B}^{d})}^p & = c_{d-3}^\mu
      \int_{\BB^3} \frac{|x_2|^p |x_3|^{n p}}{ (x_1^2+x_2^2)^\f{p}{2}} (1- x_1^2-x_2^2-x_3^2)^{\mu +\frac{d-3}{2}} \d x_1 \d x_2 \d x_3 \\
       & = c_{d-3}^\mu \int_{\BB^2} \frac{|x_2|^p}{ (x_1^2+x_2^2)^\f{p}{2}}
        (1- x_1^2-x_2^2)^{\mu +\frac{d-2}{2}+\frac{np}{2}} \d x_1 \d x_2 \\ & \qquad\qquad \times  \int_{-1}^1 |u|^{np}(1-u^2)^{\mu + \frac{d-3}{2}} \d u.
\end{align*}
The second integral on the right-hand side is evaluated by \eqref{eq:int-beta}, whereas the first integral is equal to, using the polar
coodinates $(x_2, x_3) = r (\cos \t, \sin \t)$,
$$
   \int_0^1 r (1-r^2)^{\frac{np}{2} + \mu + \frac{d-2}{2}} \d r   \int_0^{2 \pi} |\sin \t|^p \d \t
      = \frac{ \sqrt{\pi} \Gamma(\frac{p+1}{2})} {(\frac{np}{2} + \mu + \frac{d}{2}) \Gamma(\frac{p+2}{2})}.
$$
Thus, putting together, we obtain
$$
 \left\|\Phi_{1, 2} D_{1, 2} f_{1,3}^n \right\|_{L^p(W_\mu, \mathbb{B}^{d})}^p = \pi^{\frac{d-2}{2} }\frac{ \Gamma(\frac{n p+1}{2})\Gamma(\frac{p+1}{2}) \Gamma(\mu+1)}{\Gamma(\frac{p+2}{2}) \Gamma(\frac{n p +d}{2} +\mu+1)}.
$$
From these computations, the identity in \eqref{eq:exam2} follows immediately.
}
\section{Spectral Operator and $L^2$ Bernstein Inequalities on the Ball}
\setcounter{equation}{0}

In this section, we discuss Bernstein inequalities in the $L^2$ norm for the weight function $W_\mu$ on the unit ball.
The proof relies on the decomposition of the spectral operator $\CD_\mu$ in \eqref{eq:eigenB}. In the first subsection,
we utilize the decomposition in \eqref{DmuDec1} to give a new proof of the known inequalities, including the recent
result in \cite{K2023}. In the second subsection, we provide another decomposition of $\CD_\mu$, which leads to
another family of sharp Bernstein inequalities.

Throughout this section, we denote the norm of $f \in L^2(W_\mu, \BB^d)$ by $\|f\|_{\mu,2}$.

\subsection{Spectral Operator and Bernstein Inequality}

The main result in this section is the sharp Bernstein inequalities in $L^2(W_\mu, \BB^d)$ norm, stated in the following
theorem.

\begin{thm} \label{thm:B-ineq0Ball}
Let $d \ge 2$, $n = 0,1,2,\ldots$ and $f \in \Pi_n^d$. Then
\begin{equation} \label{eq:B-ineq0ball}
 \sum_{i=1}^d \left \| \sqrt{1-\|x\|^2} \partial_i f \right \|_{\mu,2}^2 +
  \sum_{1\le i<j\le d} \left \|D_{ij} f \right \|_{\mu,2}^2 \le n(n+2\mu+d) \|f \|_{\mu,2}^2
\end{equation}
and the equality holds if and only if $f \in \CV_n(W_\mu, \BB^d)$.
Furthermore, the following two inequalities are also sharp,
\begin{align}
 \sum_{i=1}^d \left \| \sqrt{1-\|x\|^2} \partial_i f \right \|_{\mu,2}^2 \le n(n+2\mu+d) \|f \|_{\mu,2}^2, \quad \text{if $n$ is even}
  \label{eq:B-ineq0aball} \\
 \sum_{i=1}^d \left \| \sqrt{1-\|x\|^2} \partial_i f \right \|_{\mu,2}^2 \le (n(n+2\mu+d) -d+1) \|f \|_{\mu,2}^2, \quad \text{if $n$ is odd}.
  \label{eq:B-ineq0aOdd}
\end{align}
\end{thm}

It should be noted that the inequalities \eqref{eq:B-ineq0aball} and \eqref{eq:B-ineq0aOdd}, as well as their sharpness,
were proved recently by A. Kro\'o \cite{K2022}. Particularly interesting is \eqref{eq:B-ineq0aOdd}, since the proof
of the other two follows more or less straightforwardly from the self-adjoint form of the differential operator $\CD_\mu$
in \eqref{eq:eigenB}, but more is needed for the proof of \eqref{eq:B-ineq0aOdd}.

The proof in \cite{K2022} is involved, in which one can see the trace of the spectral operator $\CD_\mu$, but only implicitly.
In the following, we provide an alternative proof that is based entirely on the decomposition \eqref{DmuDec1} of $\CD_\mu$
and, we believe, more intuitive. More importantly, our proof can be adopted for the stronger Bernstein inequalities that will
be discussed in the second subsection.

The essential tool for our proof is the following identity, which follows immediately from \eqref{eq:CDdecomp} by
integration by parts \cite[Theorem 2.1]{KPX},
\begin{align} \label{eq:adj-int}
  - \int_{\mathbb{B}^d} \CD_\mu f(x) g(x) W_\mu(x) \d x
     \,  = \sum_{i=1}^{d} \int_{\mathbb{B}^d}   (1-\|x\|^2) \partial_i f(x) \partial_i g(x) W_\mu(x) \d x& \\
      \quad +  \sum_{1\le i<j\le d} \int_{\mathbb{B}^d} D_{ij} f(x)  D_{ij} g(x)  W_\mu(x) \d x,& \notag
\end{align}
which implies, in particular, that $\mathcal{D}_{\mu}$ is self-adjoint. 

\bigskip\noindent
{\it Proof of Theorem \ref{thm:B-ineq0Ball}.}
Let $\l_n^\mu = -n(n+2\mu+d)$. Since $f$ is a polynomial of degree $n$ and $ \proj_j(W_\mu;f) \in \CV_n(W_\mu, \BB^d)$,
it follows from \eqref{eq:eigenB} that
$$
  f = \sum_{j=0}^n \proj_j(W_\mu;f)  \quad \hbox{and} \quad  \CD_\mu f = \sum_{j=0}^n \l_j^\mu \proj_j(W_\mu;f).
$$
Since $|\l_j^\mu | \le |\l_n^\mu |$ for $j \le n$, it follows by the Parseval identity that
\begin{align} \label{eq:pre-int_Dfg}
  \left\|\CD_\mu f \right\|_{\mu,2}^2 & = \sum_{j=0}^n (\l_j^\mu)^2  \left\|\proj_j(W_\mu;f) \right\|_{\mu,2}^2 \\
     & \le  (\l_n^{\mu} )^2 \sum_{j=0}^n  \left\|\proj_j(W_\mu;f) \right\|_{\mu,2}^2 =( \l_n^{\mu})^2 \|f\|_{\mu,2}^2. \notag
\end{align}
Consequently, by the Cauchy-Schwarz inequality, we deduce
\begin{equation} \label{eq:int_Dfg}
   \left| \int_{\mathbb{B}^d} \CD_\mu f(x)\cdot f(x) W_\mu(x) \d x \right| \le
     \left\|\CD_\mu f(x) \right\|_{\mu,2} \cdot \|f\|_{\mu,2} \le \l_n^{\mu} \|f\|_{\mu,2}^2.
\end{equation}
Setting $g = f$ in \eqref{eq:adj-int} and applying the above inequality, we have proved \eqref{eq:B-ineq0ball}, whereas
\eqref{eq:B-ineq0aball} is
an immediate consequence of \eqref{eq:B-ineq0ball}.
To see that \eqref{eq:B-ineq0aball} is sharp,  we consider the Gegenbauer polynomial
\begin{equation*}
  P_{e_1}^\mu(x) = C_n^{(\mu+\frac d 2)}(x_1),
\end{equation*}
which is the orthogonal polynomial of degree $n$ in $\CV_n(W_{\bk}, \triangle)$ by setting $\a_1 = n$ and $\a_2 \ldots = \a_{d} = 0$
for the orthogonal polynomial $P_\a(W_\mu)$ given in \eqref{secondb}.

To prove \eqref{eq:B-ineq0aOdd}, we consider polynomials $Q_{\ell, m}^{n}\left(W_{\mu}\right)$, defined in \eqref{firstb},
which form an orthogonal basis of $\CV_n(W_\mu, \BB^d)$. 
Since $D_{i,j}$ is an angular derivative, $D_{i,j} g(\|\cdot\|) = 0$ for the radius function $g(\|x\|)$, so is $\CD_\SS g(\|\cdot\|) =0$
for $\CD_\SS$ defined in \eqref{DmuDec1}. Thus,
\[
\CD_\SS  Q_{\ell, m}^{n}(x)=P_{m}^{\left(\mu, n-2 m+\frac{d-2}{2}\right)}\left(2\| x \|^{2}-1\right)
\CD_\SS Y_{\ell}^{n-2 m}(x).
\]
Because $\CD_\SS$ restricted on $\sph$ is the Laplace-Beltrami operator, by \eqref{eq:LBe}, we obtain
\[
\CD_\SS Y_{\ell}^{n-2 m} (x)= \|x\|^{n-2m} \Delta_0 Y_{\ell}^{n-2 m} (\xi)=-(n-2m)(n-2m+d-2)Y_{\ell}^{n-2m}(x).
\]
Consequently, by \eqref{eq:eigenB} and \eqref{eq:CDdecomp}, we conlude that
\[
   \mathcal{D}_{\mu}^{\mathrm{rot}}  Q_{\ell, m}^{n} = \lambda^{\mu}_{n,m}  Q_{\ell, m}^{n}, \quad 0\le \ell \le \dim \CH_{n-2m}^d,
   \quad 0 \le m \le n/2,
\]
where $\lambda^{\mu}_{n,m}= - n(n+2 \mu+d)+(n-2m)(n-2m+d-2)$. Since $|\lambda^\mu_{n,m}|$ is an increasing
function of $m$ and $0 \le m \le \frac{n-1}2$, it follows that
$$
\left |\l^{\mu}_{n,m} \right| \le n(n+2 \mu+d) - (d-1) = \left |\l^{\mu}_{n,\frac{n-1}{2}} \right|
$$
when $n$ is odd. Since $Q_{\ell, m}^{n}$ consists of an orthogonal basis of $\CV_n(W_\mu, \BB^d)$, it follows from
\eqref{eq:proj} and the Paseval identity that
\begin{align*}
   \left\|\CD_\mu^{\mathrm{rot}} \proj_n (W_\mu) \right\|_{L^2(W_\mu,\BB^d)}^2  = & \sum_{m, \ell} \left |\wh f_{\ell,m}^n\right |^2
    \left| \l_{m,n}^\mu \right | \cdot \left \| Q_{\ell, m}^{n}\right \|_{L^2(W_\mu,\BB^d)}^2  \\
    \le & \left |\l^{\mu}_{n,\frac{n-1}{2}} \right|  \sum_{m, \ell} \left |\wh f_{\ell,m}^n\right |^2
    \left \| Q_{\ell, m}^{n}\right \|_{L^2(W_\mu,\BB^d)}^2 \\
      = & \left |\l^{\mu}_{n,\frac{n-1}{2}} \right|  \left\| \proj_n (W_\mu) \right\|_{L^2(W_\mu,\BB^d)}^2.
\end{align*}
Consequently, using the identity derived from the self-adjointness of $\CD_\mu^{\mathrm{rot}}$, we can follow
the preceding proof, as given in \eqref{eq:pre-int_Dfg} and \eqref{eq:int_Dfg}, to establish the inequality \eqref{eq:B-ineq0aOdd}.
Moreover, this inequality is sharp since it becomes an identity if $f = Q_{\ell, \frac{n-1}{2}}^n$. Note that when $n$ is even,
then $\l^{\mu}_{n,\frac{n}{2}}=\l_n^\mu$, hence inequalities \eqref{eq:B-ineq0ball} and \eqref{eq:B-ineq0aball}  turn into
an equality for $Q_{\ell, \frac{n}{2}}^{n}$.
\qed

\subsection{New Decomposition of Spectral Operator and Bernstein Inequality}
In this subsection, we present another type of decomposition of the spectral operator $\mathcal{D}_{\mu}$,
which is characteristically different from \eqref{eq:CDdecomp} and is of interest in itself. It leads to several
new Bernstein inequalities, including those on the ball mentioned in the introduction.

\begin{thm}
For $d \ge 2$, the spectral operator $\CD_\mu$ on $\mathbb{B}^d$ satisfies
\begin{align} \label{dec}
\mathcal{D}_{\mu} = \frac{1}{W_\mu(x)} \left[ \frac{1}{\|x\|^d} \la x ,\nabla\ra \left(\|x\|^{d-2}(1-\|x\|^2) W_\mu(x)  \la x ,\nabla\ra \right)    \right] + \frac{1}{\|x\|^2} \CD_\SS.
\end{align}
\end{thm}

\begin{proof}
Let $r=\|x\|$. Then starting from formula~\eqref{eq:CDdecomp}, a standard computation yields:
\begin{align*}
\mathcal{D}_{\mu}&= \frac{1}{W_{\mu}(x)}
\left[
\sum_{i=1}^{d} (\mu+1)(-2x_i) W_{\mu}(x) \partial_i + W_{\mu+1}(x) \partial_i^2 \right] + \CD_\SS \\[1em]
&= -2(\mu+1) \langle x, \nabla \rangle + \bigl( 1 - \|x\|^2 \bigr) \Delta + \CD_\SS \\[1em]
&= -2(\mu+1) r \frac{\partial}{\partial r} + (1 - r^2)
\left(  \frac{\partial^2}{\partial r^2} + \frac{d-1}{r} \frac{\partial}{\partial r} + \frac{1}{r^2} \Delta_0 \right) + \Delta_0,
\end{align*}
where the last equality uses \cite[Proposition 4.1.6]{DX} and we have used $\CD_\SS = \Delta_0$. Hence,
\[
\mathcal{D}_{\mu}=(1 - r^2) \frac{\partial^2}{\partial r^2}
+ \frac{d-1}{r} \frac{\partial}{\partial r}
- (2\mu + d + 1) r \frac{\partial}{\partial r}
+ \frac{1}{r^2} \Delta_0.
\]
Now, a quick computation shows that
\begin{align*}
\la x ,\nabla\ra \left(\|x\|^{d-2}(1-\|x\|^2) W_\mu(x)  \la x ,\nabla\ra \right)  = r \frac{\partial}{\partial r} \left(r^{d-1}(1-r^2)^{\mu+1}    \frac{\partial}{\partial r} \right) \\ = r^d (1-r^2)^\mu \left( (1 - r^2) \frac{\partial^2}{\partial r^2}
+ \frac{d-1}{r} \frac{\partial}{\partial r}
- (2\mu + d + 1) r \frac{\partial}{\partial r} \right).
\end{align*}
Comparing the two identities proves the stated identity.
\end{proof}

As an application of the new decomposition of $\CD_\mu$, we obtain alternative expressions for the
integral in \eqref{eq:adj-int}.

\begin{thm} \label{thm:adj2}
Let $f$ and $g$ be functions in $C^2(\mathbb{B}^d)$. Then
\begin{align*}
  -\int_{\mathbb{B}^d} \CD_\mu f(x) \cdot g(x) W_\mu(x) \d x
     \,= &  \int_{\mathbb{B}^d} \la x ,\nabla \ra f(x) \cdot \la x ,\nabla \ra g(x) (1-\|x\|^2)W_\mu(x) \frac{\d x}{\|x\|^2} \\
      &  +  \sum_{1\le i<j\le d} \int_{\mathbb{B}^d} D_{i,j} f(x)  D_{i,j} g(x)   W_\mu(x) \frac{\d x}{\|x\|^2}. \notag
\end{align*}
\end{thm}

\begin{proof}
We apply integration on the decomposition of $\CD_\mu$ given in \eqref{dec}. For the first term in the right-hand side of
\eqref{dec}, we use the spherical-polar variable and integrate by parts on the radial variable to obtain
\begin{align*}
    \int_{\mathbb{B}^d}&  \frac{1}{\|x\|^d} \left[ \la x ,\nabla\ra \left(\|x\|^{d-2}(1-\|x\|^2) W_\mu(x)  \la x ,\nabla\ra \right)    \right] f(x) \cdot g(x) W_\mu(x) \d x \\
    & \qquad = \int_{\mathbb{S}^{d-1}} \int_{0}^{1} \frac{1}{r} \left[ r \frac{\partial}{\partial r} r^{d-2} (1-r^2)^{\mu+1} r \frac{\partial}{\partial r}  \right] f(r \xi) g(r \xi) \d r \d \sigma(\xi) \\
    & \qquad  = \int_{\mathbb{S}^{d-1}} \int_{0}^{1}    r^{d-1} (1-r^2)^{\mu+1}  \frac{\partial}{\partial r}  f(r \xi) \frac{\partial}{\partial r}g(r \xi) \d r \d \sigma(\xi) \\
    &\qquad  =\int_{\mathbb{B}^d} \frac{1-\|x\|^2}{\|x\|^2} \la x ,\nabla\ra f(x) \la x ,\nabla\ra g(x) W_\mu(x) \d x.
  \end{align*}
The spherical part follows form the fact that $D_{i,j} g(\|\cdot\|) = 0$ for the radius function $g(\|x\|)$ and $\CD_{i,j}$
are self-adjoint in $L^2(\sph)$ \cite[Proposition 1.8.4]{DaiX}.
\end{proof}

The integral identity gives another proof that $\CD_\mu$ is self-adjoint. The identity and the new decomposition of
$\CD_\mu$ in \eqref{dec} are of interests in thier own. Like their counterparts on the simplex \cite[(2.11) and (2.12)]{GX},
they are somewhat unexpected. As an immediate application of the new integral identity, we use it in place of the
identity \eqref{eq:adj-int} and follow the proof of Theorem \ref{thm:B-ineq0Ball} to derive new Bernstein inequalities in
$L^2(W_\mu, \BB^d)$. The results are stated below.

\begin{thm} \label{thm:BInequality_new}
Let $d \ge 2$, $n = 0,1,2,\ldots$ and $f \in \Pi_n^d$. Then
\begin{align} \label{eq:BB-ineq1}
     \left \|\frac{\sqrt{1-\|x\|^2}}{\|x\|}  \la x ,\nabla \ra f\right \|_{\mu,2}^2
       +  \sum_{1\le i<j\le d} \left \| \frac{1}{\|x\|} \partial_{i,j} f \right\|_{\mu,2}^2 \le
            n(n+2\mu+d) \|f\|_{\mu,2}^2,
\end{align}
and the equality holds if and only if $f \in \CV_n(W_\mu, \BB^d)$. Furthermore, the following two inequalities are also sharp,
\begin{align}
\left \|\frac{\sqrt{1-\|x\|^2}}{\|x\|}  \la x ,\nabla \ra f\right \|_{\mu,2} \le \sqrt{n(n+2\mu+d)}  \|f\|_{\bk,2}, \quad \text{if $n$ is even} \label{eq:BB-ineq1a} \\
\left \|\frac{\sqrt{1-\|x\|^2}}{\|x\|}  \la x ,\nabla \ra f\right \|_{\mu,2} \le \sqrt{n(n+2\mu+d) -d+1}  \|f\|_{\bk,2} \quad \text{if $n$ is odd}. \label{eq:BB-ineq1b}
\end{align}
\end{thm}

\begin{proof}
Using the integral identity in Theorem \ref{thm:adj2} instead of \eqref{eq:adj-int}, the proof of these inequalities follows
from that of Theorem \ref{thm:B-ineq0Ball} almost verbatim. In particular,  the polynomials attaining equalities in
\eqref{eq:BB-ineq1a} and \eqref{eq:BB-ineq1b} are the same ones for \eqref{eq:B-ineq0aball} and \eqref{eq:B-ineq0aOdd}
in Theorem \ref{thm:B-ineq0Ball}.
\end{proof}

\begin{rem}
It is worth pointing out that, in terms of the spherical-polar coordinates $x = r \xi \in \BB^d$ with $0 \le r \le 1$ and $\xi\in \sph$,
the function in the left-hand side of \eqref{eq:BB-ineq1a} and \eqref{eq:BB-ineq1b} become
$$
\frac{\sqrt{1-\|x\|^2}}{\|x\|}  \la x ,\nabla \ra f = \sqrt{1-r^2} \frac{\d f}{\d r},
$$
so that these Bernstein inequalities are compatible with the classical Bernstein inequality of one variable.
\end{rem}

We note that the inequalities \eqref{eq:BB-ineq1a} and \eqref{eq:BB-ineq1b} are different types of Bernstein inequalities
from those in \eqref{eq:B-ineq0aball} and \eqref{eq:B-ineq0aOdd}. They imply immediately inequalities for $D_{i,j}$.
While the one derived from Theorem \ref{thm:BInequality_new} is stronger because of the factor $\frac{1}{\|x\|}$, it is
weaker than the one in \eqref{pij} with $p =2$, which has the factor $\frac{1}{\sqrt{x^2_i+x^2_j}} \ge \frac{1}{\|x\|}$, if we disregard the
constant in the right-hand side.

\section*{Acknowledgment}
This work was written almost entirely while the first author was visiting
the Department of Mathematics at the University of Oregon during the fall of 2025.
He thanks the department for its warm hospitality.

{The authors thank two anonymous referees for their helpful comments, especially one referee's
suggestion of finding examples to quantify the sharpness of the new inequalities, which led us to
Examples \ref{exam1} and \ref{exam2}. }


\end{document}